\documentclass[11pt,a4paper, reqno, oneside]{amsart}

\usepackage{amsmath,amssymb,amsthm}

\numberwithin{equation}{section}
\newtheorem{theorem}{Theorem}[section]
\newtheorem{corollary}[theorem]{Corollary}

\newtheorem{lemma}[theorem]{Lemma}
\theoremstyle{remark}
\newtheorem{remark}{Remark}[section]

\theoremstyle{definition}
\newtheorem{definition}[theorem]{Definition}
\newtheorem{example}[theorem]{Example}

\newcommand{\R}{\mathbb{R}}
\newcommand{\C}{\mathbb{C}}

\begin{document}

\title[Magnetic virial. Morawetz, smoothing and Strichartz estimates]
{Magnetic virial identities, weak dispersion and Strichartz
inequalities}

\begin{abstract}
  We show a family of virial-type identities for the Schr\"odinger
  and wave equations with electromagnetic potentials. As a
  consequence, some weak dispersive inequalities in space dimension $n\geq3$,
  involving Morawetz
  and smoothing estimates, are proved; finally, we apply them to prove Strichartz
  inequalities for the wave equation with a non-trapping  electromagnetic
  potential with almost Coulomb decay.
\end{abstract}

\date{\today}

\author{Luca Fanelli}
\address{Luca Fanelli:
SAPIENZA Unversit$\grave{\text{a}}$ di Roma, Dipartimento di
Matematica, Piazzale A.~Moro 2, I-00185 Roma, Italy}
\email{fanelli@mat.uniroma1.it}

\author{Luis Vega}
\address{Luis Vega: Universidad del Pais Vasco, Departamento de
Matem$\acute{\text{a}}$ticas, Apartado 644, 48080, Bilbao, Spain}
\email{luis.vega@ehu.es}

\subjclass[2000]{35J10, 35L05, 58J45.}

\keywords{magnetic potentials, virial identities, weak dispersive
estimates, Strichartz estimates, Schr\"odinger equation, wave
equation}

\maketitle


\section{Introduction}\label{sec.introd}

In this paper, we consider electromagnetic Hamiltonians in the
standard covariant form
\begin{equation}\label{eq.H}
  H=-\nabla_A^2+V(x),
\end{equation}
where
\begin{equation}\label{eq.nablaa}
 \nabla_A=\nabla-iA,
 \qquad
 A=(A^1,\dots,A^n):\R^n\to\R^n
\end{equation}
and $V:\R^n\to\R$; the magnetic potential $A$ satisfies the Coulomb
gauge condition
\begin{equation}\label{eq.gauge}
  \text{div}A=0.
\end{equation}
Related to these Hamiltonians, we study the magnetic Schr\"odinger
equation
\begin{equation}\label{eq.schro}
  \begin{cases}
    iu_t(t,x)-Hu(t,x)=0
    \\
    u(0,x)=f(x),
  \end{cases}
\end{equation}
and the magnetic wave equation
\begin{equation}\label{eq.wave}
  \begin{cases}
    u_{tt}(t,x)+Hu(t,x)=0
    \\
    u(0,x)=f(x)
    \\
    u_t(0,x)=g(x),
  \end{cases}
\end{equation}
where in both cases the unknown is a function $u:\R^{1+n}\to\C$. We
are interested in weak dispersive phenomena for equations
\eqref{eq.schro} and \eqref{eq.wave}; in particular, we will point
our attention on the structures of $A$ and $V$ which allow the weak
dispersion.

In the family of weak dispersive estimates we find, among the
others, Morawetz and smoothing estimates.

For the free case $A\equiv 0\equiv V$, the Morawetz estimates are,
for $n\geq3$,
\begin{equation}\label{eq.morschrofree}
    \int_0^{+\infty}\int_{\R^n}\frac{\left|\partial_\tau
    e^{it\Delta}f\right|}{|x|}\leq\|f\|_{\dot H^\frac12}
\end{equation}
\begin{equation}\label{eq.morwavefree}
    \int_0^{+\infty}\int_{\R^n}\frac{\left|\partial_\tau
    e^{it\sqrt{-\Delta}}f\right|}{|x|}\leq\|f\|_{\dot H^1},
\end{equation}
where $\partial_\tau$ is the tangential derivative. Inequality
\eqref{eq.morwavefree} was proved in \cite{mor} for the Klein-Gordon
equation first, and then extended to the Schr\"odinger equation. We
also recall the smoothing estimates for the free equations
\begin{equation}\label{eq.smooschrofree}
    \sup_{R>0}\frac1R\int_0^{+\infty}\int_{|x|\leq R}
    \left|\nabla e^{it\Delta}f\right|^2\leq\|f\|_{\dot H^\frac12}
\end{equation}
\begin{equation}\label{eq.smoowavefree}
    \sup_{R>0}\frac1R\int_0^{+\infty}\int_{|x|\leq R}
    \left(\left|\partial_te^{it\sqrt{-\Delta}}f\right|^2
    +\left|\nabla e^{it\sqrt{-\Delta}}f\right|^2\right)\leq\|f\|_{\dot
    H^1}.
\end{equation}
Inequality \eqref{eq.smooschrofree} was proved independently in
\cite{cs}, \cite{s} and \cite{v}, and the proof can be easily
generalized to obtain \eqref{eq.smoowavefree}. We also mention the
paper \cite{pv}, in which  a proof of \eqref{eq.smoowavefree} for
the nonlinear wave equation based on a modification of Morawetz
ideas \cite{mor} is given. Extensions to the Schr\"odinger equation
were considered in \cite{brv} . However not so many results are
available for the magnetic case $A\neq0$; we mention \cite{brv1} \cite{pda-lf2},
\cite{ges}, \cite{GST} and \cite{ste} where Strichartz and smoothing
estimates for the magnetic Schr\"odinger and wave equations are
proved.

The aim of this paper  is to show, for non-trapping magnetic Hamiltonians, the relation
between virial identities and weak dispersion, in analogy with the
results in \cite{pv},  \cite{brv}, and \cite{brvv}.

All through the paper we will assume some regularity
assumptions on the Hamiltonian $H$.

{\bf (H1)} The Hamiltonian $H_A=-\nabla_A^2$ is essentially
self-adjoint on $L^2(\R^n)$, with form domain
\begin{equation*}
  D(H_A)=\left\{f: f\in L^2, \int|\nabla_Af|^2<\infty\right\}.
\end{equation*}

{\bf (H2)} The potential $V$ is a perturbation of $H_A$ in the
Kato-Rellich sense, i.e. there exists a small $\epsilon>0$ such that
\begin{equation}
  \|Vf\|_{L^2}\leq(1-\epsilon)\|H_Af\|_{L^2}+C\|f\|_{L^2},
\end{equation}
for all $f\in D(H_A)$.

Assumptions (H1), (H2) have several consequences about the existence
theory for equations \eqref{eq.schro}, \eqref{eq.wave}. First of
all, they imply the self-adjointness of $H$, by standard
perturbation techniques (see e.g. \cite{CFKS}); hence by the spectral
theorem we can define the Schr\"odinger and wave propagators
$S(t)=e^{itH}$, $W(t)=H^{-\frac12}e^{it\sqrt H}$. Moreover we can
define for any $s$ the distorted norms
\begin{equation*}
  \|f\|_{\dot{\mathcal H}^s}=\|H^{\frac s2}f\|_{L^2}.
  \qquad
  \|f\|_{\mathcal H^s}=\|f\|_{L^2}+\|H^{\frac s2}f\|_{L^2}.
\end{equation*}
Since $H$ and $H^s$ commute with each other, for any $s\geq0$, the
Schr\"odinger propagator $S(t)$ satisfy the family of conservation
laws
\begin{equation*}
  \|e^{itH}f\|_{\dot{\mathcal H}^s}=\|f\|_{\dot{\mathcal H}^s},
  \qquad
  s\geq0,
\end{equation*}
for all $t\in\R$. Similarly, the distorted wave energy
\begin{equation*}
  E(t)=\frac12\|u_t\|_{L^2}^2+\frac12\|u\|_{\dot{\mathcal H}^1}
\end{equation*}
is conserved on solutions of \eqref{eq.wave}.

For the validity of (H1) and (H2) see the standard reference
\cite{CFKS}.

In space dimension $n=3$ the magnetic field $B=\text{curl}A$ is a
physically relevant quantity for equations \eqref{eq.schro} and
\eqref{eq.wave}. In order to continue, we need to define the
analogous of $\text{curl}A$ in any space dimension; we give the
following definition.
\begin{definition}\label{def.B}
  For any $n\geq2$ the matrix-valued
  field $B:\R^n\to\mathcal M_{n\times n}(\R)$ is defined by
  \begin{equation*}
    B:=DA-DA^t,
    \qquad
    B_{ij}=\frac{\partial A^i}{\partial x^j}-\frac{\partial
    A^j}{\partial x^i}.
  \end{equation*}
  We also define the vector field $B_\tau:\R^n\to\R^n$ as follows:
  \begin{equation*}
    B_\tau=\frac{x}{|x|}B.
  \end{equation*}
\end{definition}
Hence $B$ is defined in terms of the anti-symmetric gradient of $A$.
In dimension $n=3$, the previous definition identifies
$B=\text{curl}\,A$, namely
  \begin{equation*}
    Bv=\text{curl}A\wedge v,
    \qquad
    \forall v\in\R^3.
  \end{equation*}
In particular, we have
  \begin{equation}\label{eq.curl}
    B_\tau=\frac{x}{|x|}\wedge\text{curl}\,A,
    \qquad n=3.
  \end{equation}
Hence $B_\tau(x)$ is the projection of $B=\text{curl}A$ on the
tangential space in $x$ to the sphere of radius $|x|$, for $n=3$.
Observe also that $B_\tau\cdot x=0$ for any $n\geq2$, hence $B_\tau$
is a tangential vector field in any dimension. Also notice that $A$
and $A+\nabla\psi$ produce the same $B$, for any $n\geq2$; actually,
in order to preserve the gauge invariance of our results, it is our
interest to give always assumptions in terms of $B$.

We can now state our main results.
\subsection{Magnetic virial identities}\label{subsec.virial}
They are convexity (in time) properties for certain relevant
quantities related to the solutions of these equations. A 3D-version
of the virial identity for the magnetic Schr\"odinger equation
appears in \cite{gon}, \cite{gon2}; in Theorem \ref{thm.virialS} we
present the generalization to any space dimension, while in Theorem
\ref{thm.virialW} we give the analogous identity for the magnetic wave
equation \eqref{eq.wave}.

We start with the Schr\"odinger equation.
\begin{theorem}[Virial for magnetic Schr\"odinger]\label{thm.virialS}
  Let $\phi:\R^n\to\R$ be a radial, real-valued multiplier,
  $\phi=\phi(|x|)$, and let
  \begin{equation}\label{eq.tetaS}
    \Theta_S(t)=\int_{\R^n}\phi|u|^2\,dx.
  \end{equation}
  Then, for
  any solution $u$ of the magnetic Schr\"odinger equation \eqref{eq.schro} with
  initial datum $f\in L^2,\ H_Af\in L^2$, the following virial-type identity holds:
  \begin{align}
    \ddot\Theta_S(t)= & 4\int_{\R^n}
    \nabla_AuD^2\phi\overline{\nabla_Au}\,dx-\int_{\R^n}|u|^2\Delta^2\phi\,dx
    \nonumber
    \\
    & -2\int_{\R^n}\phi'V_r|u|^2\,dx+4\Im\int_{\R^n}u\phi'B_\tau\cdot\overline{\nabla_Au}\,dx,
    \label{eq.virialS}
  \end{align}
  where
  \begin{equation*}
    \left(D^2\phi\right)_{jk}=\frac{\partial^2}{\partial x^j\partial
    x^k}\phi,
    \qquad
    \Delta^2\phi=\Delta(\Delta\phi),
  \end{equation*}
  for $j,k=1,\dots,n$, are respectively the Hessian matrix and the
  bi-Laplacian of $\phi$.
\end{theorem}
The analogous result for the wave equation \eqref{eq.wave} is the
following.
\begin{theorem}[Virial for magnetic wave]\label{thm.virialW}
  Let $\phi,\Psi:\R^n\to\R$, be  two radial, real-valued multipliers,
  and let
  \begin{equation}\label{eq.tetaW}
    \Theta_W(t)=\int_{\R^n}\left(\phi|u_t|^2+\phi|\nabla_Au|^2
    -\frac12(\Delta\phi)|u|^2\right)\,dx+\int|u|^2\phi V\,dx+\int|u|^2\Psi\,dx.
  \end{equation}
  Then, for any solution $u$ of the magnetic wave equation
  \eqref{eq.wave} with initial data $f,g\in L^2,\ H_Af,H_Ag\in L^2$,
  the following virial-type identity holds:
  \begin{align}
    \ddot\Theta_W(t)= & 2\int_{\R^n}\nabla_AuD^2\phi\overline{\nabla_Au}\,dx
    -\frac12\int_{\R^n}|u|^2\Delta^2\phi\,dx
    \nonumber
    \\
    &
    +2\int|u_t|^2\Psi\,dx-2\int|\nabla_Au|^2\Psi\,dx
    +\int|u|^2\Delta\Psi\,dx
    \nonumber
    \\
    &
    -\int\phi'V_r|u|^2\,dx+2\Im\int_{\R^n}u\phi'B_\tau\cdot\overline{\nabla_Au}\,dx.
    \label{eq.virialW}
  \end{align}
\end{theorem}

We give two immediate corollaries of the previous theorems.

\begin{corollary}\label{cor.var}
  Let $u$ be a solution of the magnetic Schr\"odinger
  equation \eqref{eq.schro} with $f\in L^2,\ H_Af\in L^2$. Then the {\it variance}
  \begin{equation*}
    Q(t)=\int_{\R^n}|x|^2|u|^2\,dx
  \end{equation*}
  satisfies the identity
  \begin{equation}\label{eq.variance}
    \ddot Q(t)=8\int_{\R^n}\left|\nabla_Au\right|^2\,dx
    -4\int_{\R^n}|x|V_r|u|^2\,dx
    +8\Im\int_{\R^n}|x|uB_\tau\cdot\overline{\nabla_Au}\,dx.
  \end{equation}
\end{corollary}

For the magnetic wave equation we have the following analogous
result:

\begin{corollary}\label{cor.onde}
  Let $u$ be a solution of the magnetic wave
  equation \eqref{eq.wave} with $f,g\in L^2,\ H_Af,H_Ag\in L^2$. Then the quantity
  \begin{equation*}
    Q(t)=\int_{\R^n}\left\{|x|^2\left(|u_t|^2+\left|\nabla_Au\right|^2+|u|^2V\right)-(n-1)|u|^2\right\}\,dx
  \end{equation*}
  satisfies the identity
  \begin{equation}\label{eq.ondevar}
    \ddot Q(t)=2\int_{\R^n}|u_t|^2+\left|\nabla_Au\right|^2\,dx-2\int
    V_r|u|^2\,dx
    +4\Im\int_{\R^n}uB_\tau\cdot\overline{\nabla_Au}\,dx.
  \end{equation}
\end{corollary}

The proofs of the corollaries are immediate applications of
identities \eqref{eq.virialS} and \eqref{eq.virialW} with the choice
$\phi(x)=|x|^2$, $\Psi\equiv1$.

The previous identities suggest that it is relevant to show examples
of potentials $A$ for which $B_\tau\equiv0$; we will focus our
attention on the 3D case.
\begin{example}\label{ex.coulomb}
 First we consider some singular potentials. Take
  \begin{equation}\label{eq.example}
    A=\frac1{x^2+y^2+z^2}(-y,x,0)=\frac{1}{x^2+y^2+z^2}(x,y,z)\wedge(0,0,1).
  \end{equation}
  We can check that
  \begin{equation*}
    \nabla\cdot A=0,
    \qquad
    B=-2\frac{z}{(x^2+y^2+z^2)^2}(x,y,z),
    \qquad
    B_\tau=0.
  \end{equation*}

  Another (more singular) example is the following:
  \begin{equation}\label{eq.example2}
    A=\left(\frac{-y}{x^2+y^2},\frac{x}{x^2+y^2},0\right)=
    \frac{1}{x^2+y^2}(x,y,z)\wedge(0,0,1).
  \end{equation}
  Here we have $B=(0,0,\delta)$, with $\delta$ denoting Dirac's delta function. Again we have $B_\tau=0$ .
\end{example}

\begin{example}\label{ex.biot}
  Now we show a natural generalization of the previous examples.
  Assume that $B=\text{curl}\,A:\R^3\to\R^3$ is known; since $\text{div}A=0$, we
  can reconstruct the potential $A$ using the {\it Biot-Savart}
  formula
  \begin{equation}\label{eq.BS}
    A(x)=\frac1{4\pi}\int\frac{x-y}{|x-y|^3}\wedge B(y)\,dy.
  \end{equation}
  Assume now that $B_\tau=0$, namely $x\wedge B(x)=0$; by
  \eqref{eq.BS} we have
  \begin{equation}\label{eq.BS2}
    A(x)=\frac x{4\pi}\wedge\int\frac{B(y)}{|x-y|^3}\,dy.
  \end{equation}
  To have $B_\tau=0$ it is necessary $B(y)=
  g(y)\frac{y}{|y|}$, for some scalar function $g:\R^3\to\R$. Since we want
  $A\neq0$, $g$ has not to be radial. As an example we
  consider
  \begin{equation*}
    g(y)=h\left(\frac{y}{|y|}\cdot\omega\right)|y|^{-\alpha},
  \end{equation*}
  for some fixed $\omega\in S^2$, where $h$ is homogeneous of degree
  0 and $\alpha\in\R$; consequently, the vector field $B$ is homogeneous of degree
  $-\alpha$. By \eqref{eq.BS2} we have
  \begin{equation}\label{eq.BS3}
    A(x)=\frac x{4\pi}\wedge\int\frac{h\left(\frac{y}{|y|}\cdot\omega\right)}
    {|x-y|^3|y|^\alpha}y\,dy.
  \end{equation}
  The potential $A$ is homogenous of degree $1-\alpha$, and by
  symmetry we have that $A(\omega)=0$. These examples can be easily extended to higher dimensions.
\end{example}

\subsection{Applications to dispersive estimates}\label{subsec.appl}
We pass to some applications of Theorems \ref{thm.virialS} and
\ref{thm.virialW} to weak dispersive estimates for \eqref{eq.schro}
and \eqref{eq.wave}. All the following results hold in dimension
$n\geq3$.

We need to introduce the following family of norms:
\begin{definition}\label{def.normetriple}
  For any $f:\R^3\to\R$ and $\alpha\in\R$ we define
  \begin{equation}\label{eq.normetriple}
    |||f|||_\alpha:=\int_0^{+\infty}\rho^{\alpha}\sup_{|x|=\rho}|f|\,d\rho.
  \end{equation}
\end{definition}
We state the following theorems.
\begin{theorem}[Weak dispersion for 3D Schr\"odinger]\label{thm.smooschro}
Let $n=3$; assume that
\begin{equation}\label{eq.small}
  |||B_\tau^2|||_3+|||V_r^+|||_2\leq\frac12.
\end{equation}
Then, for any solution $u$ of \eqref{eq.schro} with $f\in L^2,\
H_Af\in L^2$, the following estimate holds:
\begin{equation}\label{eq.smooschro}
  \sup_{R>0}\frac1R\int_0^{+\infty}\int_{|x|\leq R}
    \left|\nabla_Au\right|^2\,dx\,dt\leq
    C\|f\|^2_{\dot{\mathcal H}^{\frac12}}
\end{equation}
for some $C>0$. Moreover, if the strict inequality holds in
\eqref{eq.small}, we also have
\begin{align}
  & \sup_{R>0}\frac1R\int_0^{+\infty}\int_{|x|\leq R}
    \left|\nabla_Au\right|^2\,dx\,dt
    +\epsilon\int_0^{+\infty}\int_{\R^n}\frac{|\nabla_A^\tau u|^2}{|x|}\,dx\,dt
    \label{eq.smooschro+}
    \\
  &
  +\epsilon\sup_{R>0}\frac1{R^2}\int_0^{+\infty}\int_{|x|=R}|u|^2\,d\sigma\,dt
    \leq
    C\|f\|^2_{\dot{\mathcal H}^{\frac12}},
    \nonumber
\end{align}
for some $\epsilon>0$.
\end{theorem}
In higher dimension we prove the following Theorem.
\begin{theorem}[Weak dispersion for higher dimensional Schr\"odinger]\label{thm.smooschro4}
  Let $n\geq4$;
  assume that
  \begin{equation}\label{eq.ipbfin}
    |B_\tau(x)|\leq\frac{C_1}{|x|^2},
    \qquad
    |V_r^+(x)|\leq\frac{C_2}{|x|^3},
    \qquad
    C_1^2+2C_2\leq\frac23(n-1)(n-3),
  \end{equation}
  for all $x\in\R^n$. Then, for any solution of \eqref{eq.schro} with $f\in L^2,\ H_Af\in L^2$,
  the following estimate holds:
  \begin{equation}\label{eq.smoo4}
    \sup_{R>0}\frac1R\int_0^{+\infty}\int_{|x|\leq
    R}|\nabla_Au|^2\,dx\,dt
    \leq C\|f\|^2_{\dot{\mathcal H}^{\frac12}},
    \nonumber
  \end{equation}
  for some $C>0$. Moreover, if the strict inequality holds in
  \eqref{eq.ipbfin}, we also have
  \begin{align}
    & \sup_{R>0}\frac1R\int_0^{+\infty}\int_{|x|\leq
    R}|\nabla_Au|^2\,dx\,dt
    +\epsilon\int_0^{+\infty}\int_{\R^n}
    \frac{\left|\nabla_A^\tau u\right|^2}{|x|}\,dx\,dt
    \label{eq.smoo42}
    \\
    & +\epsilon\frac{(n-1)(n-3)}{2}\int_0^{+\infty}\int\frac{|u|^2}{|x|^3}\,dx\,dt\leq C\|f\|^2_{\dot{\mathcal H}^{\frac12}},
    \nonumber
  \end{align}
  for some $\epsilon>0$.
\end{theorem}
For the 3D magnetic wave equation we have the following result.
\begin{theorem}[Weak dispersion for 3D wave]\label{thm.smoowave}
  Let $n=3$, and assume that
\begin{equation}\label{eq.smallw}
  |||B\||_3+|||V_r^+|||_2\leq\frac12.
\end{equation}
Then, for any solution $u$ of \eqref{eq.wave} with $f,g\in L^2,\
H_Af,H_Ag\in L^2$, the following estimate holds:
\begin{equation}\label{eq.smoowave}
  \sup_{R>0}\frac1R\int_0^{+\infty}\int_{|x|\leq R}
    \left(|u_t|^2+\left|\nabla_Au\right|^2\right)\,dx\,dt\leq
    CE(0),
\end{equation}
where the energy $E$ is defined by
\begin{equation*}
    E(t)=\frac12\|u_t\|_{L^2}^2+\frac12\|u\|^2_{\dot{\mathcal H}^1}.
\end{equation*}
Moreover, if the strict inequality holds in \eqref{eq.small}, we
also have
\begin{align}
  & \sup_{R>0}\frac1R\int_0^{+\infty}\int_{|x|\leq R}
    \left(|u_t|^2+\left|\nabla_Au\right|^2\right)\,dx\,dt
    +\epsilon\int_0^{+\infty}\int_{\R^n}\frac{|\nabla_A^\tau u|^2}{|x|}\,dx\,dt
    \label{eq.smoowave+}
    \\
  &
  +\epsilon\sup_{R>0}\frac1{R^2}\int_0^{+\infty}\int_{|x|=R}|u|^2\,d\sigma\,dt
    \leq
    CE(0),
    \nonumber
\end{align}
for some $C>0$ and $\epsilon>0$ small.
\end{theorem}
The analogous in higher dimension is the following.
\begin{theorem}[Weak dispersion for higher dimensional wave]\label{thm.smoowave4}
  Let $n\geq4$, and assume that
  \begin{equation}\label{eq.ipbfinw}
    |B_\tau(x)|\leq\frac{C_1}{|x|^2},
    \qquad
    |V_r^+(x)|\leq\frac{C_2}{|x|^3},
    \qquad
    C_1^2+2C_2\leq\frac23(n-1)(n-3),
  \end{equation}
  for all $x\in\R^n$. Then, for any solution on \eqref{eq.schro} with $f,g\in L^2,\ H_Af,H_Ag\in L^2$,
  the following estimate holds:
  \begin{equation}\label{eq.smoow4}
    \sup_{R>0}\frac1R\int_0^{+\infty}\int_{|x|\leq R}
    \left(|u_t|^2+\left|\nabla_Au\right|^2\right)\,dx\,dt
    \leq CE(0),
  \end{equation}
  for some $C>0$. Moreover, if the strict inequality holds in
  \eqref{eq.ipbfinw}, we also have
  \begin{align}
    & \sup_{R>0}\frac1R\int_0^{+\infty}\int_{|x|\leq R}
    \left(|u_t|^2+\left|\nabla_Au\right|^2\right)\,dx\,dt
    +\epsilon\int_0^{+\infty}\int_{\R^n}
    \frac{\left|\nabla_A^\tau u\right|^2}{|x|}\,dx\,dt
    \label{eq.smoow42}
    \\
    & +\epsilon\frac{(n-1)(n-3)}{2}\int_0^{+\infty}\int_{|x|\geq
    R}\frac{|u|^2}{|x|^3}\,dx\,dt\leq CE(0),
    \nonumber
  \end{align}
  for some $\epsilon>0$.
\end{theorem}

\subsection{Strichartz estimates for the magnetic wave
equation}\label{subsec.stri} It is more or less standard to prove
Strichartz estimates as applications of Theorems \ref{thm.smoowave},
\ref{thm.smoowave4}; we do it in Theorem \ref{thm.striwave}. The key points are the estimates obtained in
the previous section and to write \eqref{eq.wave} as
\begin{equation}
  \begin{cases}\label{eq.waveF}
    u_{tt}-\Delta u=F(t,x)
    \\
    u(0)=f
    \\
    u_t(0)=g,
  \end{cases}
\end{equation}º
whereº
\begin{equation}\label{eq.F}
  F=-2iA\cdot\nabla_Au-A^2u-Vu.
\end{equation}
We
recall that a couple $(p,q)$ is said to be {\it wave admissible} if
\begin{equation}\label{eq.wavead}
  \frac2p+\frac{n-1}{q}=\frac{n-1}{2},
  \qquad
  2\leq p\leq\infty,
  \qquad
  \frac{2(n-1)}{n-3}\geq q\geq2,
  \quad
  q\neq\infty.
\end{equation}
If $(p,q)$ is a wave admissible couple, we say that it is an {\it
endpoint couple} if $p=2$. We can state the following theorem.
\begin{theorem}[Strichartz for wave]\label{thm.striwave}
  Let $n\geq3$; assume (H1), (H2) and either \eqref{eq.smallw} or \eqref{eq.ipbfinw}. Moreover  assume that
  \begin{equation}\label{eq.ipBstri}
    |B(x)|\leq\frac{C}{(1+|x|)^{2+\delta}},
    \qquad
    |V(x)|\leq\frac{C}{(1+|x|)^{2+\delta}},
  \end{equation}
  for some $C>0$ and some  $\delta>0$. If $u$ is a solution of
  \eqref{eq.wave} and $(p,q)$ is any non endpoint wave admissible
  couple, then the following Strichartz estimate holds:
  \begin{equation}\label{eq.striwave}
    \|u\|_{L^p_t\dot H^\sigma_q}\lesssim \|f\|_{\dot{\mathcal H^1}}+\|g\|_{L^2},
  \end{equation}
  where the gap of derivatives is $\sigma=\frac1q-\frac1p+\frac12$.
\end{theorem}

\begin{remark}\label{rem.schlag}
  It is interesting to compare this result with the ones in
  \cite{pda-lf2}, \cite{ges}, \cite{GST} and \cite{ste}. Assumption \eqref{eq.ipBstri}
  is made in terms of $B$; since it is the anti-symmetric gradient of $A$ and $\text{div}\,A=0$, \eqref{eq.ipBstri}
  implies that
  \begin{equation}\label{eq.ipAstri}
    A\leq\frac{C}{(1+|x|)^{1+\epsilon}}.
  \end{equation}
  It is relevant to notice that
  \eqref{eq.ipBstri} is gauge invariant, while it is
  not the same for \eqref{eq.ipAstri}. Also the non-trapping and repulsivity conditions given by \eqref{eq.smallw} and \eqref{eq.ipbfinw} imply the non-existence either of 0-resonances or eiegenvalues.
\end{remark}

\begin{remark}\label{rem.biot}
  Also notice that \eqref{eq.ipBstri}
  requires $\alpha>2$ in the homogeneous example \ref{ex.biot}.
  For these examples $A^2$ is homogeneous of degree strictly bigger than two.
From the counterexamples  given in \cite{gvv} it is natural to
expect that Strichartz estimates will fail if $\alpha<2$. Notice
that in \ref{ex.biot}, $A^2(\omega)=0=\min_{|x|=1}A^2(x)$, and this
is a necessary condition for the results in  \cite{gvv} to
hold.\end{remark}

\begin{remark}\label{rem.extend}
  It would be interesting to extend this result to the magnetic
  Schr\"odinger equation \eqref{eq.schro}. In order to do that, we
  should prove the versions of Theorems \ref{thm.smooschro},
  \ref{thm.smooschro4} with $L^2$-initial data. This will be done
  elsewhere.
\end{remark}

\section{Virial identities: proofs of Theorems \ref{thm.virialS} and \ref{thm.virialW}}
\label{sec.vir}

This section is devoted to the proofs of the
virial identities for the magnetic equations, Theorems
\ref{thm.virialS} and \ref{thm.virialW}. Let us start with the
magnetic Schr\"odinger equation.

\begin{proof}[{\bf Proof of Theorem \ref{thm.virialS}}]
Let us start by considering a solution $u\in\mathcal H^{\frac32}$ of
\eqref{eq.schro}. Using equation \eqref{eq.schro} in the form
\begin{equation}\label{eq.schro2}
  u_t=-iHu,
\end{equation}
we can easily compute
\begin{align}
  \dot\Theta_S(t) & =-i\left\langle u,[H,\phi]u\right\rangle
  \label{eq.teta1S}
  \\
  \ddot\Theta_S(t) & =-\left\langle u,[H,[H,\phi]]u\right\rangle,
  \label{eq.teta2S}
\end{align}
where the brackets $[,]$ are the commutator and the brackets
$\langle,\rangle$ are the hermitian product in $L^2$. In order to
simplify the notations, let us denote by
\begin{equation}\label{eq.T}
  T=-[H,\phi].
\end{equation}
By the Leibnitz formula
\begin{equation}\label{eq.leib}
  \nabla_A(fg)=g\nabla_Af+f\nabla g,
\end{equation}
which implies that
\begin{equation}\label{eq.leibH}
  H(fg)=(Hf)g+2\nabla_Af\cdot\nabla g+f(\Delta g),
\end{equation}
we can write explicitly
\begin{equation}\label{eq.T2}
  T=2\nabla\phi\cdot\nabla_A+\Delta\phi.
\end{equation}
Observe that $T$ is anti-symmetric, namely
\begin{equation*}
  \langle f,Tg\rangle=-\langle Tf,g\rangle;
\end{equation*}
moreover, in the case $A\equiv0$ the operator $T$ coincides with the
usual one introduce by Morawetz in \cite{mor}, which is
$2\nabla\phi\cdot\nabla+\Delta\phi$.

Hence we can rewrite \eqref{eq.teta2S} in the following form
\begin{equation}\label{eq.2teta2S}
  \ddot\Theta_S(t)=\left\langle u,[H,T]u\right\rangle,
\end{equation}
where $T$ is given by \eqref{eq.T2}.

We can compute explicitly the commutator $[H,T]$; by \eqref{eq.T2}
we have
\begin{equation}\label{eq.HT}
  [H,T]=-[\nabla_A^2,2\nabla\phi\cdot\nabla_A]-[\nabla_A^2,\Delta\phi]+[V,T]=:I+II+III.
\end{equation}
Let us introduce the following notations: for $f:\R^n\to\C$,
\begin{align*}
  f_j & =\frac{\partial f}{\partial x^j};
  \\
  f_{\widetilde j} & =f_j-iA^jf;
  \\
  f_{\widetilde j^\star} & =f_j+iA^jf.
\end{align*}
With these notations we have
\begin{equation*}
  (fg)_{\widetilde j}=f_{\widetilde j}g+fg_j;
\end{equation*}
moreover, the formula of integrations by parts is
\begin{equation*}
  \int_{\R^n}f_{\widetilde j}(x)g(x)\,dx=-\int_{\R^n}f(x)g_{\widetilde
  j^\star}(x)\,dx.
\end{equation*}
Now we compute the terms $I$, $II$ and $III$ in
\eqref{eq.HT}.

The term $III$ in \eqref{eq.HT} is easily computed:
\begin{equation}\label{eq.nuovaV}
  III=[V,T]=2[V,\nabla_A\cdot\nabla]=-2\nabla\phi\cdot\nabla V=-2\phi'V_r.
\end{equation}
For $I$, we have
\begin{align}
  -I = & 2\sum_{j,k=1}^n\left(\partial_{\widetilde j}\partial_{\widetilde j}\phi_k\partial_{\widetilde k}
  -\phi_k\partial_{\widetilde k}\partial_{\widetilde j}\partial_{\widetilde j}\right)
  \nonumber
  \\
  = & \sum_{j,k=1}^n\left(2\phi_{kjj}\partial_{\widetilde k}+4\phi_{jk}
  \partial_{\widetilde j}\partial_{\widetilde k}+2\phi_k(\partial_{\widetilde j}\partial_{\widetilde j}
  \partial_{\widetilde k}-\partial_{\widetilde k}\partial_{\widetilde j}\partial_{\widetilde
  j})\right).
  \label{eq.Iuno}
\end{align}
Notice that
\begin{align*}
  \partial_{\widetilde j}\partial_{\widetilde k}-\partial_{\widetilde k}\partial_{\widetilde j}= &
  i\left(A^j_k-A^k_j\right),
  \\
  \partial_{\widetilde j}\partial_{\widetilde j}\partial_{\widetilde k}
  -\partial_{\widetilde k}\partial_{\widetilde j}\partial_{\widetilde
  j} = &
  i\left(A^k_j-A^j_k\right)_j+2i\left(A^k_j-A^j_k\right)\partial_{\widetilde
  j};
\end{align*}
hence, by \eqref{eq.Iuno} we obtain
\begin{equation}\label{eq.Idue}
  -I=\sum_{j,k=1}^n\left(2\phi_{kjj}\partial_{\widetilde k}+4\phi_{jk}
  \partial_{\widetilde j}\partial_{\widetilde
  k}+2i\phi_j\left(A^j_k-A^k_j\right)_k+4i\phi_j\left(A^j_k-A^k_j\right)\partial_{\widetilde
  k}\right)
\end{equation}

For the term $II$ in \eqref{eq.HT} we compute
\begin{align}
  -II & =
  \sum_{j,k=1}^n\left(\partial_{\widetilde k}\partial_{\widetilde k}\phi_{jj}-
  \phi_{jj}\partial_{\widetilde k}\partial_{\widetilde k}\right)
  \nonumber
  \\
  & =\sum_{j,k=1}^n\left(\phi_{jjkk}+2\phi_{jjk}\partial_{\widetilde k}\right).
  \label{eq.II}
\end{align}
By \eqref{eq.Idue} and \eqref{eq.II} we can write
\begin{align}
  \langle u,[\nabla_A^2,T]u\rangle = &
  \sum_{j,k=1}^n\int_{\R^n}\left(2u\phi_{kjj}\overline{u_{\widetilde k}}
  +4u\phi_{jk}\overline{\partial_{\widetilde j}\partial_{\widetilde k}u}
  +2u\phi_{kjj}\overline{u_{\widetilde k}}\right)\,dx
  \nonumber
  \\
  & +\sum_{j,k=1}^n\int_{\R^n}\left(2i\phi_j\left(A^j_k-A^k_j\right)_k|u|^2
  +4iu\phi_j\left(A^j_k-A^k_j\right)\overline{u_{\widetilde
  k}}\right)\,dx
  \label{eq.III}
  \\
  & +\int_{\R^n}|u|^2\Delta^2\phi\,dx.
  \nonumber
\end{align}
Observe that
\begin{equation*}
  \overline{\partial_{\widetilde j}\partial_{\widetilde k}u}=
  \partial_{\widetilde j^\star}\partial_{\widetilde
  k^\star}\overline u;
\end{equation*}
as a consequence, integrating by parts the first three terms of
\eqref{eq.III} we have
\begin{align}
  \sum_{j,k=1}^n\int_{\R^n} & \left(2u\phi_{kjj}\overline{u_{\widetilde k}}
  +4u\phi_{jk}\overline{\partial_{\widetilde j}\partial_{\widetilde k}u}
  +2u\phi_{kjj}\overline{u_{\widetilde k}}\right)\,dx
  \label{eq.IV}
  \\
  & =\sum_{j,k=1}^n\int_{\R^n}-4u_{\widetilde j}\phi_{jk}\overline{u_{\widetilde
  k}}\,dx=-4\int_{\R^n}\nabla_A uD^2\phi\overline{\nabla_A u}\,dx.
  \nonumber
\end{align}

For the 4th and 5th term in \eqref{eq.III} we notice that
\begin{equation*}
  \sum_{j,k=1}^n\phi_{jk}\left(A^j_k-A^k_j\right)=0,
\end{equation*}
and integrating by parts we obtain
\begin{align}
  \sum_{j,k=1}^n\int_{\R^n} & \left(2i\phi_j\left(A^j_k-A^k_j\right)_k|u|^2
  +4iu\phi_j\left(A^j_k-A^k_j\right)\overline{u_{\widetilde
  k}}\right)\,dx &
  \label{eq.V}
  \\
  & =4\Im\sum_{j,k=1}^n\int_{\R^n}u\phi_j\left(A^j_k-A^k_j\right)\overline{u_{\widetilde
  k}}\,dx
  \nonumber
  \\
  & =4\Im\int_{\R^n}u\phi'B_\tau\cdot\overline{\nabla_{A}u}\,dx,
\end{align}
whit $B_tau$ as in Definition \ref{def.B}.

By \eqref{eq.nuovaV}, \eqref{eq.III}, \eqref{eq.IV} and \eqref{eq.V}
we conclude that
\begin{align}
  \langle u,[H,T]u\rangle = &
  4\int_{\R^n}\nabla_A uD^2\phi\overline{\nabla_A u}
  -\int_{\R^n}|u|^2\Delta^2\phi
  \label{eq.HT2}
  \\
  & -2\int_{\R^n}\phi'V_r|u|^2
  +4\Im\int_{\R^n}u\phi'B_\tau\cdot\overline{\nabla_{A}u}.
  \nonumber
\end{align}
Identities \eqref{eq.2teta2S} and \eqref{eq.HT2} prove
\eqref{eq.virialS}.

\begin{remark}\label{rem.reg}
  In the above arguments the highest order term in $u$ that appears is of
  the form
  \begin{equation*}
    \int\nabla_A^2u\nabla\phi\cdot\overline{\nabla_Au};
  \end{equation*}
  it makes sense thanks to assumption (H2) and the condition $f\in
  L^2,\ H_Af\in L^2$, which implies $H_Ae^{itH}f\in L^2$, and by interpolation $\nabla_Ae^{itH}f\in
  L^2$. Consequently, all the performed integration by parts  are
  permitted.
\end{remark}

\end{proof}

\begin{proof}[{\bf Proof of Theorem \ref{thm.virialW}}]

The proof of Theorem \ref{thm.virialW} is analogous to the previous
one. Let us write
\begin{equation}\label{eq.tetaW1}
  \Theta_W(t)=\langle u_t,\phi u_t\rangle+\langle\phi\nabla_A u,\nabla_A u\rangle
  -\frac12\langle u\Delta\phi,u\rangle+\langle u,\phi Vu\rangle+\langle u,\Psi u\rangle.
\end{equation}
Differentiating in \eqref{eq.tetaW1} with respect to time and using
equation \eqref{eq.wave} we obtain
\begin{align*}
  \frac{d}{dt}\langle u_t,\phi u_t\rangle & =-2\Re\langle u_t,\phi
  Hu\rangle=2\Re\langle u_t,\phi
  \nabla_A^2u\rangle-2\Re\langle u_t,\phi
  Vu\rangle,
  \\
  \frac{d}{dt}\langle\phi\nabla_A u,\nabla_A u\rangle & =-2\Re\langle u_t,\phi
  \nabla_A^2u\rangle-2\Re\langle u_t,\nabla\phi\cdot\nabla_A u\rangle,
  \\
  -\frac{d}{dt}\frac12\langle u\Delta\phi,u\rangle & =-\Re\langle
  u_t,(\Delta\phi)u\rangle,
  \\
  \frac{d}{dt}\langle u,\phi Vu\rangle & =2\Re\langle u_t,\phi
  Vu\rangle
  \\
  \frac{d}{dt}\langle u,\Psi u\rangle & =2\Re\langle u_t,\Psi u\rangle.
\end{align*}
Hence, recalling the operator
$T=-[H,\phi]=2\nabla\phi\cdot\nabla_A+\Delta\phi$, we have by
\eqref{eq.tetaW1}
\begin{equation}\label{eq.1tetaW2}
  \dot\Theta_W(t)=-\Re\langle u_t,Tu\rangle+2\Re\langle u_t,\Psi u\rangle.
\end{equation}
Consider  the first term on the RHS of \eqref{eq.1tetaW2}, differentiating
and using the equation we see that
\begin{equation}\label{eq.2tetaW}
  -\frac{d}{dt}\langle u_t,Tu\rangle=\langle
  u,HTu\rangle-\langle u_t,Tu_t\rangle.
\end{equation}
Since $T$ is anti-symmetric, we have
\begin{equation}\label{eq.1}
  \Re\langle u_t,Tu_t\rangle=0;
\end{equation}
moreover
\begin{equation*}
  \langle u,HTu\rangle=\langle
  u,THu\rangle+\langle u,[H,T]u\rangle=-
  \langle HTu,u\rangle+\langle u,[H,T]u\rangle,
\end{equation*}
and then
\begin{equation}\label{eq.2}
  \Re\langle u,HTu\rangle=\frac12\langle u,[H,T]u\rangle.
\end{equation}
Recollecting \eqref{eq.2tetaW}, \eqref{eq.1} and \eqref{eq.2} we
arrive at
\begin{equation}\label{eq.3}
  \frac{d}{dt}\Re\langle u_t,Tu\rangle=\frac12\langle
  u,[H,T]u\rangle.
\end{equation}
For the second term on the RHS of \eqref{eq.1tetaW2}, we observe
that
\begin{equation}\label{eq.simm}
  \frac{d}{dt}2\Re\langle u_t,\Psi u\rangle=2\langle u_t,\Psi
  u_t\rangle+2\Re\langle Hu,\Psi u\rangle.
\end{equation}
By integration by parts we see that
\begin{equation}\label{eq.simm2}
  \Re\langle Hu,\Psi
  u\rangle =-\int|\nabla_Au|^2\Psi-\Re\int\nabla_Au\cdot\nabla\Psi
  \overline u.
\end{equation}
Moreover,
\begin{equation*}
  \int\nabla_Au\cdot\nabla\Psi
  \overline u=-\int|u|^2\Delta\Psi-\int\overline{\nabla_Au}\cdot\nabla\Psi
  u,
\end{equation*}
and consequently
\begin{equation}\label{eq.simm3}
  \Re\int\nabla_Au\cdot\nabla\Psi
  \overline u=-\frac12\int|u|^2\Delta\Psi.
\end{equation}
In conclusion, by \eqref{eq.simm}, \eqref{eq.simm2} and
\eqref{eq.simm3} we obtain
\begin{equation}\label{eq.simmfin}
  \frac{d}{dt}2\Re\langle u_t,\Psi u\rangle=2\int|u_t|^2\Psi-2\int|\nabla_Au|^2\Psi
  +\int|u|^2\Delta\Psi.
\end{equation}

Finally, by \eqref{eq.1tetaW2}, \eqref{eq.3} and \eqref{eq.simmfin}
we conclude that
\begin{equation}\label{eq.4}
  \ddot\Theta_W(t)=\frac12\langle
  u,[H,T]u\rangle+2\int|u_t|^2\Psi-2\int|\nabla_Au|^2\Psi
  +\int|u|^2\Delta\Psi.
\end{equation}
The term $[H,T]$ on the RHS of \eqref{eq.4} has been already
computed in the previous section, modulo the constant $1/2$. The
analogous to Remark \ref{rem.reg} concludes the proof of
\eqref{eq.virialW}.

\end{proof}

\section{Proofs of the smoothing estimates}\label{sec.dim}

We devote this section to the proofs of Theorems
\ref{thm.smooschro}, \ref{thm.smooschro4}, \ref{thm.smoowave},
\ref{thm.smoowave4}. The proofs are based on suitable choices of the
multiplier $\phi$ in the virial identities \eqref{eq.virialS} and
\eqref{eq.virialW}. As we see in the following, the choice of the
multipliers is different in the cases $n=3$ and $n\geq4$, and it
follows the ideas of the paper \cite{brv}. We start with the
Schr\"odinger equation in space dimension $n=3$.

\subsection{Proof of Theorem \ref{thm.smooschro}}\label{sec.proof1}

Recalling \eqref{eq.teta2S} and \eqref{eq.T},
  let us rewrite identity \eqref{eq.virialS} as follows
  \begin{equation*}
    \frac{d}{dt}\dot\Theta_S(t) = -\langle u,[H,T]u\rangle.
  \end{equation*}
  By \eqref{eq.teta1S} and \eqref{eq.T2}, integrating by parts we
  see that
  \begin{equation*}
    \dot\Theta_S(t)=2\Im\int_{\R^n}\overline
    u(x,t)\nabla_Au(x,t)\cdot\nabla\phi(x)\,dx.
  \end{equation*}
  Hence we can rewrite
  \eqref{eq.virialS} as follows
  \begin{align}
    &  2\int_{\R^n}
    \nabla_AuD^2\phi\overline{\nabla_Au}-\frac12\int_{\R^n}|u|^2\Delta^2\phi
    -\int_{\R^n}\phi'V_r|u|^2
    \label{eq.virialS2}
    \\
    & +2\Im\int_{\R^n}u\phi'B_\tau\cdot\overline{\nabla_Au}=\mathcal R(t),
    \nonumber
  \end{align}
  for all $t>0$, where
  \begin{equation*}
  \mathcal R(t)=\frac d{dt}\left(\Im\int_{\R^n}\overline u(x,t)\nabla_Au(x,t)\cdot\nabla\phi(x)\,dx
    \right),
\end{equation*}

   We start with an interpolation Lemma that will be
  used for the estimate of the right hand side of
  \eqref{eq.virialS2}.
\begin{lemma}\label{lem.interpol}
  Let $\phi=\phi(|x|):\R^n\to\R$, $n\geq3$, be a radial function such that
  $\phi'(r)$ and $r\phi''(r)$ are bounded; then the
  following estimate holds:
  \begin{equation}\label{eq.interpol}
    \left|\int_{\R^n}\overline u(x,t)\nabla_Au(x,t)\cdot\nabla\phi(x)\,dx\right|
    \leq C\|f\|_{\dot{\mathcal H}^{\frac12}}.
  \end{equation}
\end{lemma}

\begin{proof}
  Let us consider the quadratic form
  \begin{equation*}
    T(f,g)=\int_{\R^n}\overline
    f(x)\nabla_Ag(x)\cdot\nabla\phi(x)\,dx.
  \end{equation*}
  Since $\phi'$ is bounded, we have the inequality
  \begin{equation}\label{eq.www}
    \left|T(f,g)\right|\leq C_1\|f\|_{L^2}\|\nabla_Ag\|_{L^2}.
  \end{equation}
  By integration by parts, we see that
  \begin{equation}\label{eq.qqq}
    T(f,g)=-\int_{\R^n}
    g(x)\nabla_A\overline
    f(x)\cdot\nabla\phi(x)\,dx-\int_{\R^n}\overline
    fg\Delta\phi\,dx;
  \end{equation}
  under the assumptions on $\phi$ we have that $\Delta\phi(x)\leq
  C/|x|$, hence using the magnetic Hardy's inequality \eqref{eq.hardy}
  we have by \eqref{eq.qqq} that
  \begin{equation}\label{eq.rrr}
    \left|T(f,g)\right|\leq C_2\|g\|_{L^2}\|\nabla_Af\|_{L^2}.
  \end{equation}
  By interpolation between \eqref{eq.www} and
  \eqref{eq.rrr} we get
  \begin{equation*}
    \left|T(f,g)\right|\leq C\|f\|_{\mathcal H^{1/2}}
   \|g\|_{\mathcal H^{1/2}}.
  \end{equation*}

\end{proof}

Now we choose an explicit multiplier $\phi$. For some $M>0$, let us
consider
\begin{equation*}
  \phi_0(x)=\int_0^x\phi_0'(s)\,ds,
\end{equation*}
where
\begin{equation*}
  \phi'_0=\phi'_0(r)=
  \begin{cases}
    M+\frac13r,
    \qquad
    r\leq1
    \\
    M+\frac12-\frac1{6r^2},
    \qquad
    r>1.
  \end{cases}
\end{equation*}
A direct computation shows that
\begin{equation*}
  \phi_0''(r)=
  \begin{cases}
    \frac13,
    \qquad
    r\leq1
    \\
    \frac1{3r^3},
    \qquad
    r>1
  \end{cases}
\end{equation*}
and the bilaplacian is given by
\begin{equation*}
  \Delta^2\phi_0(r)=-4\pi\delta_{x=0}-\delta_{|x|=1},
\end{equation*}
where ot the right hand side we have the Dirac masses concentrated
at zero and on the unit sphere, respectively. By scaling, for any
$R>0$ we define
\begin{equation*}
  \phi(r)=R\phi_0\left(\frac rR\right),
\end{equation*}
hence
\begin{equation}\label{eq.fi1}
  \phi'(r)=
  \begin{cases}
    M+\frac r{3R},
    \qquad
    r\leq R
    \\
    M+\frac12-\frac{R^2}{6r^2},
    \qquad
    r>R
  \end{cases}
\end{equation}
\begin{equation}\label{eq.fi2}
  \phi''(r)=
  \begin{cases}
    \frac1{3R},
    \qquad
    r\leq R
    \\
    \frac1R\cdot\frac{R^3}{3r^3},
    \qquad
    r>R
  \end{cases}
\end{equation}
\begin{equation}\label{eq.fibi}
  \Delta^2\phi(r)=-4\pi\delta_{x=0}-\frac1{R^2}\delta_{|x|=R}.
\end{equation}
Observe that the assumptions of Lemma \ref{lem.interpol} are
satisfied by $\phi$.

We can pass to the estimate on the left hand side of
\eqref{eq.virialS2}. Let us introduce the following formula, which
holds in any dimension:
\begin{equation}\label{eq.formula}
    \nabla_AuD^2\phi\overline{\nabla_Au}=
    \phi''\left|\nabla_A^ru\right|^2
    +\frac{\phi'}{r}\left|\nabla_A^\tau u\right|^2.
\end{equation}
Here $\nabla_A^\tau$ is the projection of $\nabla_A$ on the tangent
  plane to the sphere, such that
  \begin{equation*}
    \left|\nabla_A^\tau u(x)\right|^2+\left|\nabla_A^r u(x)\right|^2=
    \left|\nabla_A u(x)\right|^2,
  \end{equation*}
  \begin{equation*}
    \nabla_A^ru(x)=\left(\nabla_A u(x)\cdot\frac{x}{|x|}\right)\frac{x}{|x|},
    \qquad
    \nabla_A^\tau u\cdot\nabla_A^r u\equiv0.
  \end{equation*}
Now we insert \eqref{eq.fi1}, \eqref{eq.fi2} and \eqref{eq.fibi} in
\eqref{eq.virialS2}; neglecting the negative part $V_r^-$ of the
electric potential, we have
\begin{align}
    &  \frac{2}{3R}\int_{|x|\leq R}
    \left|\nabla_Au\right|^2\,dx
    +2M\int_{\R^n}
    \frac{|\nabla_A^\tau u|^2}{|x|}\,dx
    +\frac1{2R^2}\int_{|x|=R}|u|^2\,d\sigma
    \label{eq.virialS3}
    \\
    &
    -\int_{\R^n}\phi'V_r^+|u|^2\,dx
    +2\Im\int_{\R^n}u\phi'B_\tau\cdot\overline{\nabla_Au}\,dx
    \nonumber
    \\
    & \leq\frac d{dt}\left(\Im\int_{\R^n}\overline u(x,t)\nabla_Au(x,t)\cdot\nabla\phi(x)\,dx
    \right),
    \nonumber
  \end{align}
for all $R>0$. Let us now consider the term involving $B_\tau$ on
the left hand side of \eqref{eq.virialS3}; since the vector field
$B_\tau$ is tangential, by \eqref{eq.fi1} we can estimate
\begin{align}
  2\Im & \int_{\R^n}u\phi'B_\tau\cdot\overline{\nabla_Au}\,dx \geq
  -2\left|\Im\int_{\R^n}u\phi'B_\tau\cdot\overline{\nabla_Au}\,dx\right|
  \label{eq.radi}
  \\
  &\geq-2\left(M+\frac12\right)\int_{\R^n}|u|\cdot|B_\tau|\cdot|\nabla_A^\tau
  u|\,dx
  \nonumber
  \\
  & \geq-2\left(M+\frac12\right)\left(\int_{\R^n}\frac{|\nabla_A^\tau
  u|^2}{|x|}\,dx\right)^{\frac12}\left(\int_0^{+\infty}\,d\rho\int_{|x|=\rho}|x|\cdot|u|^2\cdot|B_\tau|^2
  \,d\sigma\right)^{\frac12}
  \nonumber
  \\
  & \geq-2\left(M+\frac12\right)K_1\left(\int_0^{+\infty}\,d\rho\left(
  \sup_{|x|=\rho}|B_\tau|^2|x|^3\right)\frac1{\rho^2}\int_{|x|=\rho}|u|^2
  \,d\sigma\right)^{\frac12}
  \nonumber
  \\
  & \geq-2\left(M+\frac12\right)K_1\left(\sup_{R>0}\frac1{R^2}\int_{|x|=R}|u|^2
  \,d\sigma\right)^{\frac12}\left(\int_0^{+\infty}\sup_{|x|=\rho}|B_\tau|^2|x|^3\,d\rho\right)^{\frac12}
  \nonumber
  \\
  & =-2\left(M+\frac12\right)K_1K_2|||B_\tau^2|||_3^{1/2},
  \nonumber
\end{align}
where
\begin{equation*}
  K_1=\left(\int_{\R^n}\frac{|\nabla_A^\tau
  u|^2}{|x|}\,dx\right)^{\frac12},
\end{equation*}
\begin{equation*}
  K_2=\left(\sup_{R>0}\frac1{R^2}\int_{|x|=R}|u|^2
  \,d\sigma\right)^{\frac12},
\end{equation*}
and we recall the definition
\begin{equation*}
  |||B_\tau^2|||_3=\int_0^{+\infty}\rho^3\sup_{|x|=\rho}|B_\tau|^2\,d\rho.
\end{equation*}
In an analogous way we treat the term involving $V_r^+$ in
\eqref{eq.virialS3}:
\begin{align}
  -\int\phi' & V_r^+|u|^2\,dx\geq
  -\left|\int\phi'V_r^+|u|^2\,dx\right|
  \label{eq.radiV}
  \\
  &
  \geq-\left(M+\frac12\right)\int_0^{+\infty}d\rho\int_{|x|=\rho}|V_r^+|\cdot|u|^2\,d\sigma
  \nonumber
  \\
  & \geq -\left(M+\frac12\right)\int_0^{+\infty}d\rho
  \left(\sup_{|x|=\rho}(|V_r^+|\cdot|x|^2)\frac1{\rho^2}\int_{|x|=\rho}|u|^2\,d\sigma\right)
  \nonumber
  \\
  & \geq-\left(M+\frac12\right)\left(\int_0^{+\infty}
  \sup_{|x|=\rho}|V_r^+|\cdot|x|^2\,d\rho\right)\cdot\left(\sup_{R>0}\frac1{R^2}
  \int_{|x|=\rho}|u|^2\,d\sigma\right)
  \nonumber
  \\
  & =-\left(M+\frac12\right)|||V_r^+|||_2K_2^2,
  \nonumber
\end{align}
where $K_2$ is as before and
\begin{equation*}
  |||V_r^+|||_2=\int_0^{+\infty}
 \rho^2 \sup_{|x|=\rho}|V_r^+\,d\rho.
\end{equation*}
Using \eqref{eq.radi}, \eqref{eq.radiV} and taking the supremum over
$R>0$ in \eqref{eq.virialS3}, we obtain
\begin{align*}
  \sup_{R>0}\frac{2}{3R} & \int_{|x|\leq R}
    \left|\nabla_Au\right|^2+2MK_1^2+\frac12K_2^2
    \\
    & \leq\mathcal
    R(t)+\left(M+\frac12\right)\cdot\left(2K_1K_2|||B_\tau^2|||_3^{1/2}+K_2^2|||V_r^+|||_2\right),
\end{align*}
or equivalently
\begin{equation}\label{eq.111}
  \sup_{R>0}\frac{2}{3R}\int_{|x|\leq R}
    \left|\nabla_Au\right|^2\,dx+
    C(M,K_1,K_2,B)\leq\mathcal
    R(t),
\end{equation}
where
\begin{align*}
  C(M,K_1,K_2,B)= & 2MK_1^2+\left[\frac12-\left(M+\frac12\right)|||V_r^+|||_2
  \right]K_2^2
  \\
  & -2\left(M+\frac12\right)|||B_\tau^2|||_3^{1/2}
  K_1K_2.
\end{align*}
Notice that all the quantities that appear in the above calculation are finite thanks to the assumptions (H1) and (H2) and the diamagnetic inequality.
In order to conclude the proof, it is sufficient now to optimize the
smallness condition on $|||B_\tau^2|||_3$ and $|||V_r^+|||_2$ under
which we can ensure that $C(M,K_1,K_2,B)\geq0 $.

Due to the homogeneity of $C$, it is not restrictive to fix $K_1=1$
and impose that
\begin{equation*}
\left[\frac12-\left(M+\frac12\right)|||V_r^+|||_2
  \right]K_2^2-2\left(M+\frac12\right)|||B_\tau^2|||_3^{1/2}
  K_2+2M>0,
\end{equation*}
for all $K_2>0$. This gives the following condition:
\begin{equation}\label{eq.ellisse}
  \frac{\left(M+\frac12\right)^2}{M}|||B_\tau^2|||_3
  +2\left(M+\frac12\right)|||V_r^+|||_2\leq1.
\end{equation}
In order to optimize \eqref{eq.ellisse} in terms of the size of $B$,
we choose $M=1/2$, such that the coefficient of $|||B_\tau^2|||_3$
is the minimum possible. Hence we obtain
\begin{equation}\label{eq.small2}
  |||B_\tau^2|||_3
  +|||V_r^+|||_2\leq\frac12\quad
  \Rightarrow
  \quad
  C(M,K_1,K_2,B)\geq0.
\end{equation}
As a consequence, if \eqref{eq.small} is satisfied we have by
\eqref{eq.111} and \eqref{eq.small2} that
\begin{equation*}
  \sup_{R>0}\frac1R\int_{|x|\leq R}
    \left|\nabla_Au\right|^2\,dx\leq C\mathcal
    R(t),
\end{equation*}
for some $C>0$. Moreover, if the strict inequality holds in
\eqref{eq.small}, we also have
\begin{align*}
  & \sup_{R>0}\frac1R\int_{|x|\leq R}
    \left|\nabla_Au\right|^2\,dx
    +\epsilon\int_{\R^n}\frac{|\nabla_A^\tau u|^2}{|x|}\,dx
    \\
  & +\epsilon\sup_{R>0}\frac1{R^2}\int_{|x|=R}|u|^2\,d\sigma
    \leq C\mathcal
    R(t),
\end{align*}
for some $\epsilon>0$.

At this point, the thesis immediately follows by integrating in time
the two last inequalities and applying Lemma \ref{lem.interpol} and
the conservation of the $\dot{\mathcal H}^{\frac12}$-norm to the
right hand side.

\subsection{Proof of Theorem \ref{thm.smooschro4}}\label{sec.proof2}

For the proof of the higher dimensional Theorem
\ref{thm.smooschro4}, we use the same techniques of the previous
one, but with different multipliers. First of all, let us start
again from \eqref{eq.virialS2}. We divide the estimate of the left
hand side into two steps, choosing two different multipliers.

{\bf Step 1.} Let us consider the multiplier $\phi(x)=|x|$, for
which

\begin{equation*}
  \phi'(r)=1,
  \qquad
  \phi''(r)=0,
  \qquad
  \Delta^2\phi(r)=-\frac{(n-1)(n-3)}{r^3}.
\end{equation*}

With this choice, by \eqref{eq.formula} we can rewrite
\eqref{eq.virialS2} as follows:

\begin{align}
  & 2\int_{\R^n}\frac{\left|\nabla_A^\tau u\right|^2}{|x|}\,dx
  +\frac{(n-1)(n-3)}{2}\int_{\R^n}\frac{|u|^2}{|x|^3}\,dx
  \label{eq.virialS5}
  \\
  &
  -\int_{\R^n}V_r|u|^2\,dx
  +2\Im\int_{\R^n}uB_\tau\cdot\overline{\nabla_A
  u}\,dx=\mathcal R(t).
  \nonumber
\end{align}

Here we used again the same notations of the previous Section. As in
the previous case, the main goal is to prove the positivity of the
left hand side. Let us assume that
\begin{equation*}
  |B_\tau(x)|\leq\frac{C_1}{|x|^2}
\end{equation*}
and estimate
\begin{align}
  -\left|2\Im\int_{\R^n}uB_\tau\cdot\overline{\nabla_A
  u}\,dx\right| &
  \geq-2\left(\int_{\R^n}\frac{|u|^2}{|x|^3}\,dx\right)^{\frac12}
  \left(\int_{\R^n}|x|^3|B_\tau|^2\left|\nabla_A^\tau u\right|^2\,dx\right)^{\frac12}
  \label{eq.virial10}
  \\
  & \geq-2C_1K_1K_2,
  \nonumber
\end{align}
where
\begin{align*}
  K_1 & =\left(\int_{\R^n}\frac{|u|^2}{|x|^3}\,dx\right)^{\frac12}
  \\
  K_2 & =\left(\int_{\R^n}\frac{\left|\nabla_A^\tau
  u\right|^2}{|x|}\,dx\right)^{\frac12}.
\end{align*}
Analogously, assume that
\begin{equation*}
  |V_r^+(x)|\leq\frac{C_2}{|x|^3}
\end{equation*}
and estimate
\begin{equation*}
  -\int_{\R^n}V_r|u|^2\,dx\geq -C_2K_1^2,
\end{equation*}
where $K_1$ is as before.

Consequently, for the left hand side of \eqref{eq.virialS5} we have
\begin{align}
  & 2\int_{\R^n}\frac{\left|\nabla_A^\tau u\right|^2}{|x|}\,dx
  +\frac{(n-1)(n-3)}{2}\int_{\R^n}\frac{|u|^2}{|x|^3}\,dx
  \label{eq.virialS6}
  \\
  & -\int_{\R^n}V_r|u|^2\,dx+2\Im\int_{\R^n}uB_\tau\cdot\overline{\nabla_A^\tau
  u}\,dx
  \nonumber
  \\
  &\ \ \geq 2K_2^2-2C_1K_1K_2-C_2K_1^2+\frac{(n-1)(n-3)}{2}K_1^2
  =:C(C_1,C_2,K_1,K_2).
  \nonumber
\end{align}
Once again, we want to optimize the condition on $C_1$ and $C_2$
under which the right hand side of \eqref{eq.virialS6} is positive
for all $K_1,K_2$.

Also here it is not restricting to fix $K_1=1$ and requiring that
\begin{equation*}
  \left[\frac{(n-1)(n-3)}{2}-C_2\right]K_1^2-2C_1K_1+2\geq0,
\end{equation*}
which gives the following condition:
\begin{equation}\label{eq.ipb}
  C_1^2+2C_2\leq(n-1)(n-3).
\end{equation}

As a consequence, if \eqref{eq.ipb} is satisfied, then
\begin{align}
  & 2\int_{\R^n}\frac{\left|\nabla_A^\tau u\right|^2}{|x|}\,dx
  +\frac{(n-1)(n-3)}{2}\int_{\R^n}\frac{|u|^2}{|x|^3}\,dx
  \label{eq.virialS7}
  \\
  & -\int_{\R^n}V_r^+|u|^2\,dx-2\Im\int_{\R^n}uB_\tau\cdot\overline{\nabla_A^\tau
  u}\,dx \geq0.
  \nonumber
\end{align}
Moreover, if the strict inequality holds in \eqref{eq.ipb}, we have
\begin{align}
  & 2\int_{\R^n}\frac{\left|\nabla_A^\tau u\right|^2}{|x|}
  +\frac{(n-1)(n-3)}{2}\int_{\R^n}\frac{|u|^2}{|x|^3}
  -\int_{\R^n}V_r^+|u|^2-2\Im\int_{\R^n}uB_\tau\cdot\overline{\nabla_A^\tau
  u}
  \label{eq.virialS8}
  \\
  & \ \ \ \geq
  \epsilon\left(\int_{\R^n}\frac{\left|\nabla_A^\tau u\right|^2}{|x|}+
  \frac{(n-1)(n-3)}{2}\int_{\R^n}\frac{|u|^2}{|x|^3}\right),
  \nonumber
\end{align}
for some $\epsilon>0$.

{\bf Step 2.} Now we perturb the multiplier $\phi$ to complete the
proof. Let us consider
\begin{equation}\label{eq.fifin}
  \widetilde\phi=\phi+\varphi,
\end{equation}
where $\phi(r)=r$ and $\varphi(r)=\int_0^r\varphi'(s)\,ds$, with
\begin{equation*}
  \varphi'(r)=
  \begin{cases}
    \frac{n-1}{2n}r,
    \qquad
    r\leq1
    \\
    \frac12-\frac1{2nr^{n-1}},
    \qquad
    r>1.
  \end{cases}
\end{equation*}
An explicit computation shows that
\begin{equation*}
  \varphi''(r)=
  \begin{cases}
    \frac{n-1}{2n},
    \qquad
    r\leq1
    \\
    \frac{n-1}{2nr^n},
    \qquad
    r>1,
  \end{cases}
\end{equation*}
\begin{equation*}
  \Delta^2\varphi=-\frac{n-1}{2}\delta_{|x|=1}-\frac{(n-1)(n-3)}{2r^3}\chi_{[1,+\infty)}.
\end{equation*}
Finally, for any $R>0$ we define the scaled multiplier
\begin{equation*}
  \widetilde\phi_R(r)=R\widetilde\phi\left(\frac rR\right);
\end{equation*}
we have explicitly that
\begin{equation}\label{eq.fifin2}
  \widetilde\phi_R(r)=r+R\varphi\left(\frac rR\right),
\end{equation}
where
\begin{equation}\label{eq.fifin3}
  \varphi'_R(r)=
  \begin{cases}
    \frac{(n-1)r}{2nR},
    \qquad
    r\leq R
    \\
    \frac12-\frac R{2nr^{n-1}},
    \qquad
    r>R.
  \end{cases}
\end{equation}
\begin{equation}\label{eq.fifin4}
  \varphi''_R(r)=
  \begin{cases}
    \frac1R\cdot\frac{n-1}{2n},
    \qquad
    r\leq R
    \\
    \frac1R\cdot\frac{R^n(n-1)}{2nr^n},
    \qquad
    r>R,
  \end{cases}
\end{equation}
\begin{equation}\label{eq.fifin5}
  \Delta^2\varphi_R=-\frac{n-1}{2R^2}\delta_{|x|=R}-\frac{(n-1)(n-3)}{2r^3}\chi_{[R,+\infty)}.
\end{equation}
At this point, we put the multiplier $\widetilde\phi_R$ in
\eqref{eq.virialS2}. Observe that, in this case,
\begin{equation*}
  \sup_{r\geq0}\widetilde\phi'(r)=\frac32;
\end{equation*}
hence for the terms involving $V_r$ and $B_\tau$ we can repeat the
same computation in \eqref{eq.virial10}, with the scaled constants
$\widetilde C_1=\frac32C_1$, $\widetilde C_2=\frac32C_2$; in this
way, the final condition on $B_\tau$ and $V_r$ turns out to be
\eqref{eq.ipbfin}. Under this assumption, when we put
$\widetilde\phi_R$ into \eqref{eq.virialS2}, we can use the term
$|x|$ in \eqref{eq.fifin2} to control the perturbative terms with
the potentials. Finally, by \eqref{eq.virialS7}, \eqref{eq.formula},
\eqref{eq.fifin3}, \eqref{eq.fifin4}, \eqref{eq.fifin5}, we have
proved that
\begin{equation}\label{eq.virialS9}
    \sup_{R>0}\frac1R\int_{|x|\leq R}|\nabla_Au|^2\,dx
    \leq C\mathcal R(t),
\end{equation}
for some $C>0$, if \eqref{eq.ipbfin} is satisfied. Moreover, if the
strict inequality holds in \eqref{eq.ipbfin}, we also have
\begin{equation}\label{eq.virialS11}
  \epsilon\left(\int_{\R^n}\frac{\left|\nabla_A^\tau u\right|^2}{|x|}\,dx
  +\frac{(n-1)(n-3)}{2}\int_{\R^n}\frac{|u|^2}{|x|^3}\,dx\right)
  \leq C\mathcal R(t),
\end{equation}
analogously to \eqref{eq.virialS8}. Now the proof continues exactly
the one of the 3D Theorem \ref{thm.smooschro}, by integration in
time and application of Lemma \ref{lem.interpol}.

\subsection{Proofs of Theorems \ref{thm.smoowave} and \ref{thm.smoowave4}.}
The proofs of the smoothing Theorems for the magnetic equation are
identical to the ones for the Schr\"odinger equations. The starting
point is now the virial identity \eqref{eq.virialW}, which can be
written as follows:
\begin{align}
    &  2\int_{\R^n}
    \nabla_AuD^2\phi\overline{\nabla_Au}\,dx-\frac12\int_{\R^n}|u|^2\Delta^2\phi\,dx
    \label{eq.virialW2}
    \\
    &
    +2\int_{\R^n}|u_t|^2\Psi\,dx-2\int_{\R^n}|\nabla_Au|^2\Psi\,dx+\int_{\R^n}|u|^2\Delta\Psi\,dx
    \nonumber
    \\
    & -\int_{\R^n}\phi'V_r|u|^2\,dx
    +2\Im\int_{\R^n}u\phi'B_\tau\cdot\overline{\nabla_Au}\,dx=\mathcal R(t),
    \nonumber
\end{align}
where
\begin{equation}\label{eq.radonde}
  \mathcal R(t)=-\frac{d}{dt}\Re\int_{\R^n}u_t(2\nabla\phi\cdot\overline{\nabla_A}u+\overline
  u\Delta\phi+\overline u\Psi)\,dx.
\end{equation}
For the LHS of \eqref{eq.virialW2}, we use the same multiplier
$\phi$ of the Schr\"odinger theorems, while the choice of $\Psi$ is
the following:
\begin{equation*}
  \Psi(x)=
  \begin{cases}
    \frac1{2R},
    \qquad
    |x|\leq R
    \\
    0,
    \qquad
    |x|>R.
  \end{cases}
\end{equation*}
By direct computation we see that
\begin{equation*}
  \Delta\Psi=
  \begin{cases}
    \frac nR,
    \qquad
    |x|\leq R
    \\
    \frac{n-1}{|x|},
    \qquad
    |x|>R,
  \end{cases}
\end{equation*}
hence $\Delta\Psi\geq0$ and the term involving it in
\eqref{eq.virialW2} can be neglected.
The only thing to control is the positivity of the term
\begin{equation*}
  \int_{|x|\leq
  R}(\nabla_AuD^2\phi\overline{\nabla_Au}-|\nabla_Au|^2\Psi)\,dx,
\end{equation*}
which is ensured by the choice of the constant $1/2$ in the
definition of $\Psi$ and the explicit formulas for $\phi'$ and
$\phi''$ introduced in the previous sections.

Finally, after integration in time, with the same techniques
involving Cauchy-Schwartz, magnetic Hardy's inequality and energy
conservation, the RHS of \eqref{eq.virialW2} turns out to be
controlled by $CE(0)$, and this concludes the proof.

\section{Strichartz estimates for the magnetic wave
equation}\label{sec.stri}

The final section of this paper is devoted to the proof of Theorem
\ref{thm.striwave}, as application of the smoothing estimates,
Theorems \ref{thm.smoowave} and \ref{thm.smoowave4}. We start with a
preliminary Lemma.

\begin{lemma}\label{lem.leving}
  Let $(p,q)$ be a non endpoint wave admissible couple; then, for all $T\in\R$, the
  following estimate holds
  \begin{equation}\label{eq.leving}
    \left\|\int_0^T\frac{\sin\left((T-\tau)\sqrt{-\Delta}\right)}{\sqrt{-\Delta}}
    F(\tau,\cdot)\,d\tau\right\|_{L^p\dot H^\sigma_q}\lesssim
    \sum_{j\in\mathbb Z}2^{\frac j2}\|F_j\|_{L^2L^2},
  \end{equation}
  where $\sigma=\frac1q-\frac1p+\frac12$ and
  \begin{equation*}
    F_j(t,x)=
    \begin{cases}
      F(t,x),
      \qquad
      |x|\in[2^j,2^{j+1}]
      \\
      0,
      \qquad
      |x|\in[0,\infty)\setminus[2^j,2^{j+1}].
    \end{cases}
  \end{equation*}
\end{lemma}
\begin{proof}
  We recall the usual Strichartz estimate for the free wave equation
  (see \cite{gv}, \cite{kt})
  \begin{equation}\label{eq.freestri}
    \left\|(\sqrt{-\Delta})^{-1}e^{it\sqrt{-\Delta}}f\right\|_{L^p\dot
    H^\sigma_q}\lesssim\|f\|_{L^2},
  \end{equation}
  with $\sigma=\frac1q-\frac1p+\frac12$ and the inequality
  \begin{equation}\label{eq.leving2}
    \sup_{R>0}\int_0^\infty\int_{|x|\leq
    R}\left|e^{it\sqrt{-\Delta}}f\right|^2\,dx\,dt\lesssim\|f\|_{L^2},
  \end{equation}
  analogous to \eqref{eq.smoowavefree}.
  By \eqref{eq.freestri} we get
  \begin{align}
    \left\|\int_0^T\frac{\sin\left((t-\tau)\sqrt{-\Delta}\right)}{\sqrt{-\Delta}}
    F(\tau,\cdot)\right\|_{L^p\dot H^\sigma_q}
    & \leq
    \left\|\frac{e^{it\sqrt{-\Delta}}}{\sqrt{-\Delta}}
    \int_0^Te^{-i\tau\sqrt{-\Delta}}F(\tau,\cdot)\right\|_{L^p\dot
    H^\sigma_q}
    \label{eq.ttstar}
    \\
    &
    \leq\left\|\int_0^Te^{-i\tau\sqrt{-\Delta}}F(\tau,\cdot)\,d\tau\right\|_{L^2}.
    \nonumber
  \end{align}
  The dual of estimate \eqref{eq.leving2} is
  \begin{equation*}
    \left\|\int
    e^{-i\tau\sqrt{-\Delta}}F(\tau,\cdot)\,d\tau\right\|_{L^2}\lesssim
     \sum_{j\in\mathbb Z}2^{\frac j2}\|F_j\|_{L^2L^2}
  \end{equation*}
  (see e.g. \cite{pv}); as a consequence and by a standard application of the Christ-Kiselev Lemma (\cite{ck}), we
  have \eqref{eq.leving}.
\end{proof}
Now we pass to the proof of Theorem \ref{thm.striwave}. As observed
in Section \ref{subsec.stri}, we can rewrite \eqref{eq.wave} in the
form \eqref{eq.waveF}, with $F$ given by \eqref{eq.F}. The solution
of \eqref{eq.waveF} is represented by
\begin{equation}\label{eq.solwave}
  u(t,\cdot)=\cos(t\sqrt{-\Delta})f+\frac{\sin(t\sqrt{-\Delta})}{\sqrt{-\Delta}}g+
  \int_0^t\frac{\sin\left((t-\tau)\sqrt{-\Delta}\right)}{\sqrt{-\Delta}}F(\tau,\cdot)\,d\tau.
\end{equation}
For the first two terms  we use  \eqref{eq.freestri}. For the last
term, recalling \eqref{eq.F}, by \eqref{eq.leving} we estimate:
\begin{align}\label{eq.anabla}
  & \left\|\int_0^t\frac{\sin\left((t-\tau)\sqrt{-\Delta}\right)}{\sqrt{-\Delta}}
    \left(A\cdot\nabla_Au+A^2u+Vu\right)\,d\tau\right\|_{L^p\dot H^\sigma_q}
    \lesssim
    \\
    &\lesssim\sum_{j\in\mathbb Z}2^{\frac
    j2}\left\{\left(\int_0^t\int_{|x|\in[2^j,2^{j+1}]}A\cdot\nabla_Au\right)^{\frac12}+
    \left(\int_0^t\int_{|x|\in[2^j,2^{j+1}]}(A^2+V)u\right)^{\frac12}\right\}.
    \nonumber
\end{align}
As observed in Remark \ref{rem.schlag}, assumption
\eqref{eq.ipBstri} implies
\begin{equation*}
  |A(x)|\leq\frac{C}{(1+|x|)^{1+\epsilon}},
\end{equation*}
moreover, \eqref{eq.ipBstri} is compatible with \eqref{eq.smallw}
and \eqref{eq.ipbfinw}. Consequently, by Holder inequality and the
smoothing estimates \eqref{eq.smoowave+} and \eqref{eq.smoow4} we
obtain
\begin{align}\label{eq.anabla2}
  & \sum_{j\in\mathbb Z}2^{\frac
    j2}\|(A\cdot\nabla_Au)_j\|_{L^2L^2}\leq
    \\
    &\leq
    \left(\sum_{j<0}2^{\frac j2}+\sum_{j\geq0}2^{\frac j2}2^{j\left(-\frac12-\epsilon\right)}
    \right)\left(\sup_{j\in\mathbb
    Z}\frac1{2^j}\int_0^t\int_{2^j|x|<2^{j+1}}|\nabla_Au|^2\right)^{\frac12}
    \leq C\sqrt{E(0)};
    \nonumber
    \\
  & \sum_{j\in\mathbb Z}2^{\frac j2}\|(A^2u+Vu)_j\|_{L^2L^2}\leq
  \label{eq.aquadrovu}
  \\
  &\leq\left(\sum_{j<0}2^{\frac j2}+\sum_{j\geq0}2^{-2j\epsilon}
    \right)\left(\sup_{j\in\mathbb
    Z}\frac1{2^{3j}}\int_0^t\int_{2^j|x|<2^{j+1}}|u|^2\right)^{\frac12}
    \nonumber
    \\
    &\sim C\left(\sup_{j\in\mathbb
    Z}\frac1{2^{2j}}\int_0^t\int_{|x|=2^j}|u|^2\,d\sigma\right)^{\frac12}
    \leq C\sqrt{E(0)}.
    \nonumber
\end{align}
By \eqref{eq.solwave}, \eqref{eq.anabla}, \eqref{eq.anabla2} and
\eqref{eq.aquadrovu} the proof is complete. \hfill$\square$

\appendix

\section{Magnetic Hardy's inequality}\label{sec.app}

We devote an Appendix to the magnetic version of Hardy's Inequality.
\begin{theorem}\label{thm.hardy}
  Let $n\geq3$ and let $A:\R^n\to\R^n$, $\nabla_A=\nabla-iA$. Then, for any $f\in D(H_A)$
  the following inequality
  holds:
  \begin{equation}\label{eq.hardy}
    \int_{\R^n}\frac{|f|^2}{|x|^2}\,dx\leq\frac{4}{(n-2)^2}\int_{\R^n}|\nabla_Af|^2\,dx.
  \end{equation}
\end{theorem}

\begin{proof}
  We only need to prove \eqref{eq.hardy} for $f\in\mathcal
  C^\infty_0$, then we conclude by density.
  Let us observe that, for all $\alpha\in\R$,
  \begin{align}
    0 & \leq\int_{\R^n}|\nabla_Af+\alpha\frac{x}{|x|^2}f|^2\,dx
    \label{eq.ardi1}
    \\ &
    =\int_{\R^n}|\nabla_Af|^2\,dx+\alpha^2\int_{\R^n}\frac{|f|^2}{|x|^2}\,dx
    +2\alpha\Re\int_{\R^n}\overline
    f\frac{x}{|x|^2}\cdot\nabla_Af\,dx.
    \nonumber
  \end{align}
  By integration by parts, using the Leibnitz formula
  \begin{equation*}
    \nabla_A(fg)=g\nabla_Af+f\nabla g,
  \end{equation*}
  we see that
  \begin{equation*}
    2\alpha\Re\int_{\R^n}\overline
    f\frac{x}{|x|^2}\cdot\nabla_Af\,dx=-\alpha\int_{\R^n}
    |f|^2\text{div}\,\frac{x}{|x|^2}\,dx=-(n-2)\alpha
    \int_{\R^n}
    \frac{|f|^2}{|x|^2}\,dx.
  \end{equation*}
  Using this in \eqref{eq.ardi1} we have
  \begin{equation*}
    \left\{-\alpha^2+(n-2)\alpha\right\}\int_{\R^n}
    \frac{|f|^2}{|x|^2}\,dx\leq\int_{\R^n}
    |\nabla_Af|^2\,dx,
  \end{equation*}
  for all $\alpha\in\R$. Now we observe that
  \begin{equation*}
    \max_{\alpha\in\R}\left\{-\alpha^2+(n-2)\alpha\right\}=\frac{(n-2)^2}{4},
  \end{equation*}
  and this completes the proof.
\end{proof}


\end{document}